\documentclass[extended]{MyStyle}            % onecolumn (standard format)
\smartqed  % flush right qed marks, e.g. at end of proof
\usepackage{graphicx}
\usepackage{amssymb}
\usepackage{amssymb,amsmath,color,soul,enumerate}

\numberwithin{equation}{section}
\numberwithin{proposition}{section}
\numberwithin{theorem}{section}
\numberwithin{definition}{section}
\numberwithin{example}{section}
\numberwithin{lemma}{section}
\numberwithin{corollary}{section}

\newcommand{\ri}{\rightarrow}

\newcommand{\RR}{{\mathbb R}}

\journalname{}
\begin{document}

\title{Coupled systems of nonlinear variational inequalities and applications \thanks{This research was partially supported by CNCS-UEFISCDI Grant No. PN-III-P1-1.1-TE-2019-0456.}
}

\titlerunning{Coupled systems of nonlinear variational inequalities and applications}        % if too long for running head

\author{\normalsize{\rm Nicu\c sor} {\sc Costea}$^\ast$}

\authorrunning{N. Costea} % if too long for running head

\institute{
              {$^\ast${\scriptsize  Department of Mathematics and Computer Science,   {\sc Politehnica} University of Bucharest, 313 Splaiul Independen\c tei, 060042 Bucharest, Romania \\ 
 }
             \email{{\tt nicusorcostea@yahoo.com;  nicusor.costea2606@upb.ro}}  }           
}             
%\date{\today}

\maketitle

{\small
\begin{abstract}   In this paper we investigate the existence of solutions for  a system consisting of two inequalities of variational type. Each inequality is formulated in terms of a nonlinear bifunction $\chi$ and $\psi$, respectively  and a coupling functional $B$. We consider two sets of assumptions ${\bf (H_\chi^i)}$, ${\bf (H_\psi^j)}$ and ${\bf (H_B^k)}$, $i,j,k\in\{1,2\}$ and we show that, if the constraints sets are bounded, then a solution exists regardless if we assumed the first or the second hypothesis on $\chi$, $\psi$ or $B$, thus obtaining eight possibilities. When the constraint sets are unbounded a coercivity condition is needed to ensure the existence of solutions. We provide two such conditions. We consider nonlinear coupling functionals, whereas, in all the papers that we are aware of that dealing with such type of inequality systems the coupling functional is assumed bilinear and satisfies a certain "inf-sup" condition. An application, arising from Contact Mechanics,  in the form of a partial differential inclusion driven by the $\Phi$-Laplace operator is presented in the last section.
\\

\end{abstract}
}

\section{Introduction}

Let $X, {Y}$ be two real Banach spaces and $K\subseteq X$, $\Lambda\subseteq  {Y}$ be nonempty, closed and convex  subsets. Assume $B:X\times {Y}\rightarrow \RR$, $\chi:X\times X\rightarrow \RR$, $\psi: {Y}\times {Y}\rightarrow \RR$, $f\in X^\ast$ and $g\in {Y}^\ast$ are given and consider the following {\it coupled system of nonlinear variational inequalities}:

Find $(u,\lambda)\in K\times\Lambda$ such that 
$$
(S):\ \; \left\{ 
\begin{array}{ll}
B(v, \lambda)-B(u, \lambda)+\chi(u,v-u)\geq \langle f,v-u\rangle, & \ \forall v\in K,\\
B(u, \lambda)-B(u,\mu)+\psi( \lambda,\mu- \lambda)\geq \langle g,\mu- \lambda \rangle ,& \ \forall \mu\in \Lambda.
\end{array}
\right.
$$

Let us consider the following {\it saddle point problem}:
 
Find $(u, \lambda)\in X\times\Lambda$ such that 
$$
(SP):\ \; \left\{ 
\begin{array}{ll}
a(u,v)+b(v,\lambda) = (f,v), & \ \forall v\in X,\\
 b(u,\mu-\lambda)\leq 0,& \ \forall \mu\in \Lambda,
\end{array}
\right.
$$
where $X$ and $Y$ are Hilbert spaces, $\Lambda\subset Y$,  $f\in X$, $a:X\times X\rightarrow \mathbb{R}$ is a bilinear continuous coercive form, $b:X\times Y \rightarrow \mathbb{R}$ is a bilinear continuous form satisfying the following {\it "inf-sup property"}
\begin{equation}\label{inf-sup Prop}
\exists \alpha>0: \ \; \inf\limits_{\stackrel{\mu\in Y}{\mu\neq 0_Y}} \sup\limits_{\stackrel{v\in X}{v\neq 0_X}}\frac{b(v,\mu)}{\|v\|_X \|\mu\|_Y}\geq \alpha.
\end{equation}
It is well-known (see, e.g., \cite{Ekeland-Temam99,Has-Hla-Necas96}) that problem  $(SP)$ possesses a uniques solution  which is exactly the saddle point of the corresponding {\it energy functional} 
$\mathcal{E}:X\times \Lambda\rightarrow \mathbb{R}$, 
$
\mathcal{E}(v,\mu):=\frac{1}{2}a(v,v)+b(v,\mu)-(f,v),
$
i.e.
$$
\mathcal{E}(u,\mu) \leq \mathcal{E}(u,\lambda) \leq \mathcal{E}(v,\lambda), \ \; \forall v\in X,\ \forall \mu\in\Lambda,
$$
hence the name for $(SP)$.

Over the recent years several generalizations of $(SP)$ have been considered in connection to the weak solvability of unilateral frictionless or bilateral frictional contact problems for linearly elastic materials. We briefly present some of these generalizations below.

\begin{enumerate}[$(a)$]
\item Find $(u, \lambda)\in X\times\Lambda$ such that 
$$
(S_1):\ \; \left\{ 
\begin{array}{ll}
\langle A(u),v \rangle+b(v,\lambda) = \langle f,v\rangle, & \ \forall v\in X,\\
 b(u,\mu-\lambda)\leq 0,& \ \forall \mu\in \Lambda,
\end{array}
\right.
$$
where $X,Y$ are reflexive Banach spaces, $f\in X^\ast$, $A:X\rightarrow X^\ast$ is a nonlinear operator and $b:X\times Y\rightarrow \mathbb{R}$ is a bilinear continuous form satisfying \eqref{inf-sup Prop}. Existence of solutions for $(S_1)$ was established by Matei \cite{Matei-RWA14} under the assumption that $A$ is hemicontinuous and satisfies a generalized monotonicity property who then used the theoretical results to derive the weak solvability of contact problems model the antiplane shear deformation of cylindrical bodies in frictional contact with a rigid foundation.

\item Find $(u, \lambda)\in X\times\Lambda$ such that 
$$
(S_2):\ \; \left\{ 
\begin{array}{ll}
\langle A(u),v-u \rangle+b(v-u,\lambda)+J^0(\gamma u;\gamma v-\gamma u) \geq \langle f,v-u\rangle, & \ \forall v\in X,\\
 b(u,\mu-\lambda)\leq 0,& \ \forall \mu\in \Lambda,
\end{array}
\right.
$$
where $X,Y$ are reflexive Banach spaces, $Z$ is a Banach space, $f\in X^\ast$, $A:X\rightarrow X^\ast$ is a nonlinear operator, $J:Z\rightarrow \mathbb{R}$ is locally Lipschitz, $\gamma:X\rightarrow Z$ is a linear and continuous operator and $b:X\times Y\rightarrow \mathbb{R}$ is again a bilinear continuous form satisfying \eqref{inf-sup Prop}. Existence and uniqueness results were established by Bai, Mig\'{o}rski \& Zeng \cite{Bai-Mig-Zeng20} and various applications to contact mechanics were provided in \cite{Mig-Bai-Zeng19}.

\item Find $(u, \lambda)\in X\times\Lambda$ such that 
$$
(S_3):\ \; \left\{ 
\begin{array}{ll}
J(v)-J(u)+b(v-u,\lambda)+\varphi(v)-\varphi(u) \geq ( f,v-u), & \ \forall v\in X,\\
 b(u,\mu-\lambda)\leq 0,& \ \forall \mu\in \Lambda,
\end{array}
\right.
$$
where $X,Y$ Hilbert spaces, $f\in X$, $J,\varphi:X\rightarrow [0,\infty)$ are convex and lower semicontinuous functionals  and $b:X\times Y\rightarrow \mathbb{R}$ is bilinear, continuous and satisfies \eqref{inf-sup Prop}. Problem $(S_3)$ was investigated by Matei \cite{Matei15} under the assumptions that $J$ and $\varphi$ satisfy appropriate boundedness conditions.

\end{enumerate}
 It is easy to check that systems $(S_1)-(S_3)$ are particular cases of $(S)$; for example it suffices to choose $K:=X$, $B(v,\mu):=b(v,\mu)$, $\chi(u,v):=\langle A(u),v\rangle +J^0(\gamma u;\gamma v) $, $\psi(\lambda,\mu):=0$ and $g:=0_{Y^\ast}$ in order to get $(S_2)$. However, in all the above examples it is assumed $b$ is bilinear and continuous whereas we do not impose neither linearity, nor the "inf-sup property" on the functional $B$, thus allowing the study of fully nonlinear coupled systems. Moreover, the above systems consist either of two variational inequalities or a hemivariational inequality and a variational one. By adding the functional $\psi$ we can study (as particular cases) systems consisting of two hemivariational (or even quasi-hemivariational) inequalities and this can turn out useful also for applications to contact mechanics as it will allow larger classes of frictional problems to be modelled. 
 
Assume $A:X\rightarrow X^\ast$ and $F:Y\rightarrow Y^\ast$ are nonlinear operators, $h:X\rightarrow [0,\infty)$, $J:Z_1\rightarrow \mathbb{R}$ and $G:Z_2\rightarrow \mathbb{R}$ are locally Lipschitz functionals and $\gamma:X\rightarrow Z_1$, $i:Y\rightarrow Z_2$ are linear and continuous operators. If we choose $\chi(u,v):=\langle A(u), v\rangle +h(u)J^0(\gamma u;\gamma v)$ and $\psi(\lambda,\mu):=\langle F(\lambda),\mu\rangle+G^0(i\lambda;i\mu)$, then $(S)$ becomes the following coupled system of a quasi-hemivariational inequality and a hemivariational inequality:
 
Find $(u,\lambda)\in K\times\Lambda$ such that 
$$
(S_4):\ \; \left\{ 
\begin{array}{ll}
\langle A(u),v-u \rangle + B(v, \lambda)-B(u, \lambda)+h(u)J^0(\gamma u;\gamma v-\gamma u)\geq \langle f,v-u\rangle, & \ \forall v\in K,\\
\langle F(\lambda),\mu-\lambda\rangle +B(u, \lambda)-B(u,\mu)+G^0(i\lambda;i \mu-i\lambda)\geq \langle g,\mu- \lambda \rangle ,& \ \forall \mu\in \Lambda.
\end{array}
\right.
$$
To our best knowledge systems of type $(S_4)$ have not yet been studied. Moreover, very little can be found in the literature  even in the {\it decoupled case}, i.e.,  when $B\equiv 0$.

The following two theorems will play a key role in proving the main results. The first theorem is a refinement of Ky Fan's minimax principle due to Brezis, Nirenberg and Stampacchia while second is a variational alternative due to Mosco.
\begin{theorem}[Brezis, Nirenberg \& Stampacchia \cite{B-N-S72}]\label{B-N-S}
Let $\mathcal{K}$ be a nonempty convex subset of a Hausdorff topological (real) vector space and $h:\mathcal{K}\times \mathcal{K}\rightarrow \RR$ a function satisfying:
\begin{enumerate}[$(a)$]

\item  $h(x,x)\leq 0$ for each $x\in \mathcal{K}$;

\item  for each $x\in \mathcal{K}$, the set $\{y\in \mathcal{K}:\ h(x,y)> 0\} $ is convex;

\item  for each $y\in \mathcal{K}$, $x\mapsto h(x,y)$ is lower semicontinuous on the intersection of $\mathcal{K}$ with any finite dimensional subspace of $E$;

\item  whenever $x,y\in \mathcal{K}$ and $\{x_\alpha\}$ is a net converging to $x$, then $h(x_\alpha,(1-t)x+ty)\leq 0$ for all $t\in [0,1]$ implies $h(x,y)\leq 0$;

\item  there exist a compact subset  $\mathcal{K}_0$ of $E$  and $y_0\in \mathcal{K}_0\cap \mathcal{K}$ such that $h(x,y_0)>0$ for all $x\in \mathcal{K}\setminus \mathcal{K}_0$.
\end{enumerate}
Then there exists a point $x_0\in \mathcal{K}_0\cap \mathcal{K}$ such that 
 $$
 h(x_0,y)\leq 0, \ \; \forall y\in \mathcal{K}.
 $$
In particular, 
$
\inf\limits_{x\in \mathcal{K}}\sup\limits_{y\in \mathcal{K}}h(x,y)\leq 0.
$
\end{theorem}

\begin{theorem}[Mosco \cite{mosco}]\label{Mosco}
Let $\mathcal{K}$ be a nonempty, compact and convex subset of a topological
vector space $E$ and $\varphi:E\ri (-\infty,\infty]$ be a proper,
convex and lower semicontinuous functional such that ${\cal
D}(\varphi)\cap \mathcal{K}\neq\emptyset$. Assume $T,U:E\times E\ri\RR$ are
two functions that satisfy:
\begin{enumerate}[$(a)$]
\item  $U(x,y)\leq T(x,y), \mbox{ for all } x,y\in E$;

\item  $x\mapsto T(x,y)$ is a concave mapping for each $y\in E$;

\item  $y\mapsto U(x,y)$ is a lower semicontinous mapping for
each $x\in E$.

\end{enumerate}
Then for each $a\in\mathbb{R}$  the following alternative holds:
\begin{description}
\item $\bullet$ either there exists $y_0\in {\cal D}(\varphi)\cap \mathcal{K}$ such that $U(x,y_0)+\varphi(y_0)-\varphi(x)\leq a$,
 $\mbox{ for all } x\in E$,

\item or,

\item $\bullet$ there exists $x_0\in E$ such that $T(x_0,x_0)>a$.
\end{description}
\end{theorem}

\section{A wide variety of existence results}

In this section we  establish the existence of at least one  solution for system $(S)$, first under the assumption that $K$ and $\Lambda$ are bounded, then  we impose additional coercivity conditions to derive the existence of solutions when at least one of the sets is unbounded. The first set of assumptions is given below.

\begin{description}
\item ${\bf (H_B^1)}$ $B:X\times Y\rightarrow \mathbb{R}$ is a functional such that:

\begin{enumerate}[$(i)$]
\item for each $\lambda\in Y$, the mapping  $u\mapsto B(u,\lambda)$ is convex and lower semicontinuous;

\item for each $u\in X$, the mapping $\lambda \mapsto B(u,\lambda)$ is concave and upper semicontinuous;
\end{enumerate}

\medskip

\item ${\bf (H_\chi^1)}$ $\chi:X\times X\rightarrow \mathbb{R}$ is a  functional such that:
\begin{enumerate}[$(i)$]
 
 \item for each $v\in X$, the mapping $u\mapsto \chi(u,v-u)$ is weakly upper  semicontinuous;
 
 \item for each $u\in X$, the mapping $v\mapsto\chi(u,v)$ is convex; 
   
 \item $\chi(u,0_X)= 0$ for all $u\in X$.

\end{enumerate}

\medskip

\item ${\bf (H_\psi^1)}$ $\psi:Y\times Y\rightarrow \mathbb{R}$ is a functional such that:
\begin{enumerate}[$(i)$]
   \item for each $\mu\in Y$, the mapping $\lambda\mapsto \psi(\lambda,\mu-\lambda)$ is weakly upper  semicontinuous;
   
   \item for each $\lambda\in Y$, the mapping $\mu\mapsto\psi(\lambda,\mu)$ is convex; 
  
  \item $\psi(\lambda,0_Y)=0$ for all $\lambda\in Y$.
\end{enumerate}

\end{description}

\begin{lemma}\label{BoundedCase1}
Assume $X$ and $Y$ are real reflexive Banach spaces and $K\subset X$ and $\Lambda \subset Y$ are nonempty, bounded, closed and convex subsets. If conditions ${\bf (H_B^1),\ (H_\chi^1)}$ and ${\bf (H_\psi^1)}$ hold, then for any pair  $(f,g)\in X^\ast\times Y^\ast$ the system  $(S)$ possesses at least one solution.
\end{lemma}

\begin{proof}
Let $E:=X\times Y$ and $\mathcal{K}:=K\times \Lambda\subset E$.  Here and hereafter, we denote the elements of $E$ as $[u,\lambda]$. Since $K$ and $\Lambda$ are bounded, closed and convex it follows that $\mathcal{K}$ is weakly compact.  We define the functional $h:\mathcal{K}\times\mathcal{K}\rightarrow \mathbb{R}$ by 
$$
h([u,\lambda],[v,\mu]):=\langle f,v-u\rangle+\langle g,\mu-\lambda\rangle -\chi(u,v-u)-\psi(\lambda,\mu-\lambda)+B(u,\mu)-B(v,\lambda),
$$
and we prove next that this functional satisfies the conditions of the Brezis-Nirenberg-Stampacchia minimax principle (see Theorem \ref{B-N-S}). We have 
$$
h([u,\lambda],[u,\lambda])=0, \ \;\forall [u,\lambda]\in\mathcal{K},
$$
hence $(a)$ of Theorem \ref{B-N-S} holds. In order to prove $(b)$ we fix $[u,\lambda]\in \mathcal{K}$ and assume $t\in[0,1]$ and $[v_i,\mu_i]\in C([u,\lambda])$, $i=1,2$, where
$$
C([u,\lambda]):=\left\{ [v,\mu]\in \mathcal{K}:\  h([u,\lambda],[v,\mu])>0\right\}.
$$
Thus,
\begin{align*}
h([u&,\lambda], t[v_1,\mu_1]+(1-t)[v_2,\mu_2])=\langle f,tv_1+(1-t)v_2 -u\rangle+\langle g,t\mu_1+(1-t)\mu_2-\lambda\rangle\\
&-\chi(u,tv_1+(1-t)v_2-u)-\psi(\lambda,t\mu_1+(1-t)\mu_2-\lambda)+B(u,t\mu_1+(1-t)\mu_2)\\
&-B(tv_1+(1-t)v_2,\lambda)\geq t\left[ \langle f,v_1-u\rangle \langle g,\mu_1-\lambda\rangle-\chi(u,v_1-u)+B(u,\mu_1)-B(v_1,\sigma)\right]\\
&+(1-t)\left[ \langle f,v_2-u\rangle \langle g,\mu_2-\lambda\rangle -\chi(u,v_2-u)\right. \left.+B(u,\mu_2)-B(v_2,\sigma)\right]>0,
\end{align*}
which shows that $t[v_1,\mu_1]+(1-t)[v_2,\mu_2]\in C([u,\lambda])$, i.e., $C([u,\lambda])$ is a convex subset of $\mathcal {K}$. 

Now, let us fix $[v,\mu]\in \mathcal{K} $ and assume $[u_n,\lambda_n]\rightharpoonup [u,\lambda]$ as $n\rightarrow \infty$. Then hypotheses ${\bf (H_B^1)}$, ${\bf (H_\chi^1)}$, ${\bf (H_\psi^1)}$ ensure the following estimates hold
$$
\liminf\limits_{n\rightarrow \infty} \langle f,v-u_n\rangle =\langle f,v-u\rangle \mbox{ and } \liminf\limits_{n\rightarrow \infty} \langle g,\mu-\lambda_n\rangle =\langle g,\mu-\lambda\rangle,
$$
$$
\liminf\limits_{n\rightarrow \infty} (-\chi( u_n,v- u_n))=-\limsup\limits_{n\rightarrow \infty} \chi( u_n,v- u_n)\geq -\chi( u,v-u),
$$
$$
\liminf\limits_{n\rightarrow \infty} (-\psi( \lambda_n,\mu- \lambda_n))=-\limsup\limits_{n\rightarrow \infty} \psi( \lambda_n,\mu- \lambda_n)\geq -\psi( \lambda,\mu- \lambda),
$$
and $$
\liminf\limits_{n\rightarrow \infty} B(u_n, \mu)\geq B(u, \mu) \mbox{ and }
\liminf\limits_{n\rightarrow \infty} (-B(v, \lambda_n))=-\limsup\limits_{n\rightarrow \infty} (B(v, \lambda_n)\geq B(v, \lambda).
$$
This means that for each $[v,\mu]\in \mathcal{K}$ the mapping $[u,\lambda]\mapsto h([u,\lambda],[v,\mu])$ is weakly lower semicontinuous on $\mathcal{K}$, hence conditions $(c)$ is automatically  fulfilled.  Moreover, if $[u,\lambda], [v,\mu]\in\mathcal{K}$ are fixed and $[u_\alpha,\lambda_\alpha]$ is a net in $\mathcal{K}$ such that $[u_\alpha,\lambda_\alpha]\rightharpoonup [u,\lambda]$ and 
$$
h([u_\alpha,\lambda_\alpha],[(1-t)u+tv,(1-t)\lambda+t\mu])\leq 0, \ \; \forall t\in [0,1],
$$
then for any $t\in [0,1]$ one has
$$
0\geq \liminf\limits_{\alpha} h([u_\alpha,\lambda_\alpha],[(1-t)u+tv,(1-t)\lambda+t\mu])\geq h([u,\lambda],[(1-t)u+tv,(1-t)\lambda+t\mu]).
$$
Choosing $t:=1$ in the previous relation we infer that $(d)$ also holds. In order order to prove the last condition of Theorem \ref{B-N-S} fix $[v_0,\mu_0]\in \mathcal{K}$ and define 
$$
\mathcal{K}_0:=\left\{ [u,\lambda]\in\mathcal{K}:\ h([u,\lambda],[v_0,\mu_0])\leq0\right\}.
$$
The set $\mathcal{K}_0$ turns out nonempty as $[v_0,\mu_0]\in\mathcal{K}_0$ and weakly closed  due to the weakly lower semicontinuity of $[u,\lambda]\mapsto h([u,\lambda], [v_0,\mu_0])$. Consequently, $\mathcal{K}_0$ is weakly compact and 
$$
h([u,\lambda],[v_0,\mu_0])>0,\mbox{ for all } [u,\lambda]\in \mathcal{K}\setminus\mathcal{K}_0.
$$
Applying Theorem \ref{B-N-S} for $E$ endowed with the weak topology (keep in mind that this is a Hausdorff topological vector space) we get the existence of an element $[u_0,\lambda_0]\in \mathcal{K}_0$ such that 
\begin{equation}\label{SolBound}
h([u_0,\lambda_0],[v,\mu])\leq 0,\mbox{ for all }[v,\mu]\in \mathcal{K}.
\end{equation}
In order to complete the proof it suffices to show that $[u_0,\lambda_0]$ solves $(S)$. Taking $\mu:=\lambda_0$ in \eqref{SolBound} we get
$$
\langle f,v-u_0\rangle-\chi(u_0,v-u_0)+B(u,\lambda_0)-B(v,\lambda_0)\leq 0,\mbox{ for all }v\in K,
$$
which is exactly the first inequality of system $(S)$. The second inequality of $(S)$ is obtained by taking $v:=u_0$ in \eqref{SolBound}.
\qed
\end{proof}

Now let us consider a second set of assumptions on the functionals $B$, $\chi$ and $\psi$. By combining the two sets of assumptions we obtain various existence results for our inequality system $(S)$. 

\begin{description}
\item ${\bf (H_B^2)}$ $B:X\times Y\to \mathbb{R}$ is a functional such that:

\begin{enumerate}[$(i)$]
\item for each $\mu\in Y$ the mapping $[u,\lambda]\mapsto 2B(u,\mu)-B(u,\lambda)$ is weakly lower semicontinuous;

\item for each $v\in X$ the mapping $[u,\lambda]\mapsto 2B(v,\lambda)-B(u,\lambda)$ is concave.
\end{enumerate} 

\bigskip

\item ${\bf (H_\chi^2)}$ $\chi:X\times X\rightarrow \mathbb{R}$ is a functional such that:
\begin{enumerate}[$(i)$]

 \item $\chi(u,u-v)+\chi(v,v-u)\geq 0$ for all $u,v\in X$;
 
 \item for each $u,v,w\in X$, the mapping $[0,1]\ni t\mapsto \chi(u+t(v-u),w)$ is continuous at $0_+$;
 
 \item for each $u\in X$,  $v\mapsto\chi(u,v)$ is concave, upper semicontinuous and positive homogeneous; 
 
 \item $\chi(u,0_X)=0$ for all $u\in X$.
\end{enumerate}

\medskip

\item ${\bf (H_\psi^2)}$ $\psi:Y\times Y\rightarrow \mathbb{R}$ is a functional such that:
\begin{enumerate}[$(i)$]

 \item $\psi(\lambda,\lambda-\mu)+\psi(\mu,\mu-\lambda)\geq 0$ for all $\lambda,\mu \in Y$;
 
 \item for each $\lambda,\mu,\sigma\in Y$, the mapping $[0,1]\ni t\mapsto \psi(\lambda+t(\mu-\lambda),\sigma)$ is continuous at $0_+$;
 
 \item for each $\lambda\in Y$,  $\mu\mapsto\psi(\lambda,\mu)$ is concave, upper semicontinuous and positive homogeneous; 
 
 \item $\psi(\lambda,0_Y)=0$ for all $\lambda\in Y$.
\end{enumerate}

\end{description}

\begin{lemma}\label{BoundedCase2}
Assume $X$ and $Y$ are real reflexive Banach spaces and $K\subset X$ and $\Lambda \subset Y$ are nonempty, bounded, closed and convex subsets. Assume in addition that  ${\bf (H_B^i),\ (H_\chi^j)}$ and ${\bf (H_\psi^k)}$ hold for $i,j,k\in\{1,2\}$. Then for any pair  $(f,g)\in X^\ast\times Y^\ast$ the system  $(S)$ possesses at least one solution.
\end{lemma}

\begin{proof}
We already proved the case when ${\bf (H_B^1),\ (H_\chi^1)}$ and ${\bf (H_\psi^1)}$ are fulfilled, therefore we need to consider the remaining cases.

Let $E:=X\times Y$ and $\mathcal{K}:=K\times \Lambda$ and define $\varphi: E\rightarrow (-\infty,\infty]$ by
$$
\varphi([u,\lambda]):=I_{\mathcal{K}}([u,\lambda])-\langle f,u\rangle-\langle g,\lambda\rangle, 
$$
where $I_{\mathcal{K}}$ is the {\it indicator function} of $\mathcal{K}$, i.e., 
$$
I_{\mathcal{K}}([u,\lambda]):=\left\{
\begin{array}{ll}
0, & \mbox{ if }[u,\lambda]\in\mathcal{K},\\
\infty,&\mbox{ otherwise}.
\end{array}
\right.
$$
Since $\mathcal{K}$ is nonempty, convex and closed, it follows that $\mathcal{K}$ is weakly compact and $\varphi$ is proper, convex and lower semicontinuous and thus weakly lower semicontinuous and $\mathcal{D}(\varphi)=\mathcal{K}$.

\begin{description}

\item {\sc Case 1.} ${\bf (H_B^2),\ (H_\chi^1)}$ and ${\bf (H_\psi^1)}$ hold.

Define $U: E\times E \to\mathbb{R}$ by 
$$
U([v,\mu],[u,\lambda]):=2B(u,\mu)-B(u,\lambda)-B(v,\mu)-\chi(u,v-u)-\psi(\lambda,\mu-\lambda).
$$
Then, $U$ is concave with respect to $[v,\mu]$ and weakly lower semicontinuous with respect to $[u,\lambda]$. Moreover, $U([v,\mu],[v,\mu])=0$ for all $[v,\mu]\in E$. Consequently, we can apply Mosco's Alternative (see Theorem \ref{Mosco}) with $T:=U$, $E$ endowed with the weak topology and $a:=0$ to get the existence of $[u_0,\lambda_0]\in \mathcal{K}$ such that 
$$
U([v,\mu],[u_0,\lambda_0])+\varphi([u_0,\lambda_0])-\varphi([v,\mu])\leq 0,\ \; \forall [v,\mu]\in E,
$$
which is equivalent to 
\begin{equation}\label{BoundedEstimate2}
2B(u_0,\mu)-B(u_0,\lambda_0)-B(v,\mu)-\chi(u_0,v-u_0)-\psi(\lambda_0,\mu-\lambda_0)+\langle f, v-u_0\rangle+\langle g,\mu-\lambda_0\rangle\leq 0, 
\end{equation}
for all $[v,\mu]\in \mathcal{K}$. 

Choosing $\mu:=\lambda_0$ in \eqref{BoundedEstimate2} one has 
$$
B(u_0,\lambda_0)-B(v,\lambda_0)-\chi(u_0,v-u_0)+\langle f,v-u_0 \rangle \leq 0, \ \; \forall v\in K, 
$$
while for $v:=u_0$ the inequality \eqref{BoundedEstimate2} reduces to 
$$
B(u_0,\mu)-B(u_0,\lambda_0)-\psi(\lambda_0,\mu-\lambda_0)+\langle g,\mu-\lambda_0\rangle\leq 0, \ \;\forall \mu\in\Lambda,
$$
i.e., $[u_0,\lambda_0]$ solves $(S)$.

\item {\sc Case 2.} ${\bf (H_B^1),\ (H_\chi^2)}$ and ${\bf (H_\psi^2)}$ hold.

Define $T,U: E\times E \to\mathbb{R}$ by 
$$
T([w,\sigma],[u,\lambda]):=\chi(u,u-w)+\psi(\lambda,\lambda-\sigma)+B(u,\sigma)-B(w,\lambda),
$$
and
$$
U([w,\sigma],[u,\lambda]):=-\chi(w,w-u)-\psi(\sigma,\sigma-\lambda)+B(u,\sigma)-B(w,\lambda).
$$

Then $[w,\sigma]\mapsto T([w,\sigma],[u,\lambda])$ is concave, while   $[u,\lambda]\mapsto U([w,\sigma],[u,\lambda])$ is weakly lower semicontinuous,. Moreover,
$$
T([w,\sigma],[w,\sigma])=0, \ \; \forall [w,\sigma]\in E,
$$
and 
$$
T([w,\sigma],[u,\lambda])-U([w,\sigma],[u,\lambda])=\chi(u,u-w)+\chi(w,w-u)+\psi(\lambda,\lambda-\sigma)+\psi(\sigma,\sigma-\lambda)\geq 0,
$$
for all $[u,\lambda],[w,\sigma]\in E$. Consequently, we can apply Mosco's Alternative  for $E$ endowed with the weak topology and $a:=0$ to get the existence of $[u_0,\lambda_0]\in \mathcal{K}$ such that 
$$
U([w,\sigma],[u_0,\lambda_0])+\varphi([u_0,\lambda_0])-\varphi([w,\sigma])
\leq 0,\ \; \forall [w,\sigma]\in E.
$$
Thus,
\begin{equation}\label{BoundedSolEst2}
-\chi(w,w-u_0)-\psi(\sigma,\sigma-\lambda_0)+B(u_0,\sigma)-B(w,\lambda_0)+\langle f,w-u_0\rangle+\langle g,\sigma-\lambda_0\rangle\leq 0,
\end{equation}
for all $[w,\sigma]\in\mathcal{K}$.

Let $[v,\mu]\in\mathcal{K}$ and $t\in (0,1)$ be fixed.  Choosing $[w,\sigma]:=[u_0+t(v-u_0),\lambda_0]$ in \eqref{BoundedSolEst2} we get
$$
-\chi(u_0+t(v-u_0),t(v-u_0))+B(u_0,\lambda_0)-B(u_0+t(v-u_0),\lambda_0)+t\langle f,v-u_0\rangle\leq 0,
$$
which leads to
$$
t\langle f,v-u_0 \rangle\leq t\chi(u_0+t(v-u_0),v-u_0)-B(u_0,\lambda_0)+tB(v,\lambda_0)+(1-t)B(u_0,\lambda_0).
$$
Dividing by $t>0$, then letting $t\to 0_+$ we get the first inequality of $(S)$.  

For the second inequality it suffices to choose $[w,\sigma]:=[u_0,\lambda_0+t(\mu-\lambda_0)]$ in \eqref{BoundedSolEst2}.

\item {\sc Case 3.} ${\bf (H_B^1),\ (H_\chi^1)}$ and ${\bf (H_\psi^2)}$ hold.

Define $T, U: E\times E\to \mathbb{R}$ by 
$$
T([w,\sigma],[u,\lambda]):=B(u,\sigma)-B(w,\lambda)-\chi(u,w-u)+\psi(\lambda,\lambda-\sigma),
$$ 
and 
$$
U([w,\sigma],[u,\lambda]):=B(u,\sigma)-B(w,\lambda)-\chi(u,w-u)-\psi(\sigma,\sigma-\lambda),
$$
and apply Moscos's Alternative to get the existence of $[u_0,\lambda_0]\in\mathcal{K}$ such that 
\begin{equation}\label{BoundedEstimate4}
B(u_0,\sigma)-B(w,\lambda_0)-\chi(u_0,w-u_0)-\psi(\sigma,\sigma-\lambda_0)+\langle f,w-u_0\rangle+\langle g,\sigma-\lambda_0\rangle\leq 0,
\end{equation}
for all $[w,\sigma]\in\mathcal{K}.$

Let $[v,\mu]\in \mathcal{K}$ be fixed. Choosing  $[w,\sigma]:=[v,\lambda_0]$ in \eqref{BoundedEstimate4} we get the first inequality of $(S)$, while for $[w,\sigma]:=[u_0,\lambda_0+t(\mu-\lambda_0)]$, $t\in (0,1)$, we have 
\begin{eqnarray*}
0 &\geq& B(u_0,\lambda_0+t(\mu-\lambda_0))-B(u_0,\lambda_0)-\psi(\lambda_0+t(\mu-\lambda_0),t(\mu-\lambda_0))+\langle g,t(\mu-\lambda_0)\rangle\\
	&\geq& t\left[ B(u_0,\mu)-B(u_0,\lambda_0)-\psi(\lambda_0+t(\mu-\lambda_0),\mu-\lambda_0)+\langle g,\mu-\lambda_0\rangle\right]. 
\end{eqnarray*}
Dividing by $t>0$, then letting $t\to 0_+$ we get the second inequality of $(S)$.

\item {\sc Case 4.} ${\bf (H_B^1),\ (H_\chi^2)}$ and ${\bf (H_\psi^1)}$ hold.

Define $T, U: E\times E\to \mathbb{R}$ by 
$$
T([w,\sigma],[u,\lambda]):=B(u,\sigma)-B(w,\lambda)+\chi(u,u-w)-\psi(\lambda,\sigma-\lambda),
$$ 
and 
$$
U([w,\sigma],[u,\lambda]):=B(u,\sigma)-B(w,\lambda)-\chi(w,w-u)-\psi(\lambda,\sigma-\lambda),
$$
and follow the same steps as in the previous case.

\item {\sc Case 5.} ${\bf (H_B^2),\ (H_\chi^2)}$ and ${\bf (H_\psi^2)}$ hold.

Define $T, U: E\times E\to \mathbb{R}$ by 
$$
T([w,\sigma],[u,\lambda]):=2B(u,\sigma)-B(w,\sigma)-B(u,\lambda)+\chi(u,u-w)+\psi(\lambda,\lambda-\sigma),
$$ 
and 
$$
U([w,\sigma],[u,\lambda]):=2B(u,\sigma)-B(w,\sigma)-B(u,\lambda)-\chi(w,w-u)-\psi(\sigma,\sigma-\lambda),
$$
and apply Mosco's Alternative to get the existence of $[u_0,\lambda_0]\in\mathcal{K}$ such that 
\begin{equation}\label{BoundedEstimate5}
2B(u_0,\sigma)-B(w,\sigma)-B(u_0,\lambda_0)-\chi(w,w-u_0)-\psi(\sigma,\sigma-\lambda_0)+\langle f,w-u_0\rangle+\langle g,\mu-\lambda_0\rangle\leq 0, 
\end{equation}
for all $ [w,\sigma]\in \mathcal{K}$.

Let $[v,\mu]\in \mathcal{K}$ and recall that $H([w,\sigma]):=2B(u_0,\sigma)-B(w,\sigma)$ is concave, therefore for any $t\in(0,1)$ one has
$$
H\left(t[v,\lambda_0]+(1-t)[u_0,\lambda_0]\right)\geq t H([v,\lambda_0])+(1-t)H([u_0,\lambda_0]),
$$
and 
$$
H(t[u_0,\mu]+(1-t)[u_0,\lambda_0])\geq tH([u_0,\mu])+(1-t)H([u_0,\lambda_0]),
$$
i.e., 
\begin{equation}\label{Concavity1}
2B(u_0,\lambda_0)-B(u_0+t(v-u_0),\lambda_0)\geq (1+t)B(u_0,\lambda_0)-tB(v,\lambda_0),
\end{equation}
and 
\begin{equation}\label{Concavity2}
B(u_0,\lambda_0+t(\mu-\lambda_0))\geq tB(u_0,\mu)+(1-t)B(u_0,\lambda_0),
\end{equation}
respectively.

Taking $[w,\sigma]:=[u_0+t(v-u_0),\lambda_0]$ in \eqref{BoundedEstimate5} and keeping \eqref{Concavity1} in mind we get 
\begin{eqnarray}
\nonumber 0&\geq& 2B(u_0,\lambda_0)-B(u_0+t(v-u_0))-B(u_0,\lambda_0)+\langle f,t(v-u_0)\rangle \\
\nonumber &&-\chi(u_0+t(v-u_0),t(v-u_0))\\
\label{Ineq5.1} &\geq &t \left[ B(u_0,\lambda_0)-B(v,\lambda_0)-\chi(u_0+t(v-u_0),v-u_0)+\langle f,v-u_0\rangle \right].
\end{eqnarray}
On the other hand, kaking $[w,\sigma]:=[u_0,\lambda_0+t(\mu-\lambda_0)]$ in \eqref{BoundedEstimate5} and using 
\eqref{Concavity2} we get 
\begin{eqnarray}
\nonumber 0&\geq& B(u_0,\lambda_0+t(\mu-\lambda_0))-B(u_0,\lambda_0)-\psi(\lambda_0+t(\mu-\lambda_0),t(\mu-\lambda_0))+\langle g,t(\mu-\lambda_0)\rangle \\
\label{Ineq5.2} &\geq &t \left[ B(u_0,\mu)-B(u_0,\lambda_0)-\psi(\lambda_0+t(\mu-\lambda_0),\mu\lambda_0)+\langle f,v-u_0\rangle \right]
\end{eqnarray}
Dividing \eqref{Ineq5.1} and \eqref{Ineq5.2}, then letting $t\to 0_+$ we get the desired inequalities.

\item {\sc Case 6.} ${\bf (H_B^2),\ (H_\chi^1)}$ and ${\bf (H_\psi^2)}$ hold.

Define $T, U: E\times E\to \mathbb{R}$ by 
$$
T([w,\sigma],[u,\lambda]):=2B(u,\sigma)-B(w,\sigma)-B(u,\lambda)-\chi(u,w-u)+\psi(\lambda,\lambda-\sigma),
$$ 
and 
$$
U([w,\sigma],[u,\lambda]):=2B(u,\sigma)-B(w,\sigma)-B(u,\lambda)-\chi(u,w-u)-\psi(\sigma,\sigma-\lambda),
$$
and follow the same steps as in Case 4.

\item {\sc Case 7.} ${\bf (H_B^2),\ (H_\chi^2)}$ and ${\bf (H_\psi^1)}$ hold.

Define $T, U: E\times E\to \mathbb{R}$ by 
$$
T([w,\sigma],[u,\lambda]):=2B(u,\sigma)-B(w,\sigma)-B(u,\lambda)+\chi(u,u-w)-\psi(\lambda,\sigma-\lambda),
$$ 
and 
$$
U([w,\sigma],[u,\lambda]):=2B(u,\sigma)-B(w,\sigma)-B(u,\lambda)-\chi(w,w-u)-\psi(\lambda,\sigma-\lambda),
$$
and follow the same steps as in Case 4.
\end{description}
\qed
\end{proof}

If the sets $K$ and $\Lambda$ are unbounded, then we need to impose a coercivity condition in order to prove the existence of solutions.  Two such conditions are provided below.

\begin{description}
\item  ${\bf (C_1)}$ ${\displaystyle \frac{\chi(u,-u)+\psi(\lambda,-\lambda)}{\sqrt{\|u\|_X^2+\|\lambda\|_Y^2 }} \rightarrow -\infty}$ as $ \sqrt{\|u\|_X^2+\|\lambda\|_Y^2 }\rightarrow\infty$;

\medskip

\item ${\bf {\bf (C_2)}}$ There exist $m_\chi,m_\psi>0$ and $p,q\geq 1$ such that:
\begin{enumerate}[$(i)$]
\item $\chi(u,-u)\leq m_\chi \|u\|_X^p$ for all $u\in X$;

\smallskip

\item $\psi(\lambda,-\lambda)\leq m_\psi \|\lambda\|_Y^q$ for all $\lambda\in Y$;
\smallskip

\item ${\displaystyle \frac{B(0_X,\lambda)-B(u,0_Y)}{\sqrt{\|u\|_X^2+\|\lambda\|_Y^2 }^{\max\{p,q\}} }\to -\infty}$ as $\sqrt{\|u\|_X^2+\|\lambda\|_Y^2 }\to\infty$.
\end{enumerate} 
\end{description}

The main result of the paper is given by the following theorem. Note that there are twelve possible cases to choose from depending whether we impose ${\bf (C_1)}$ or ${\bf (C_2)}$, ${\bf (H_B^1)}$ or ${\bf (H_B^2)}$ and so on.

\begin{theorem}\label{Unbounded1} Suppose $X$ and $Y$  are real reflexive Banach spaces and let $0_X\ni K\subseteq X$, $0_Y\ni \Lambda \subseteq Y$  be unbounded, closed and convex subsets. Assume in addition that either ${\bf (C_1)}$, ${\bf (H_B^1)}$, ${\bf (H_\chi^j)}$ and ${\bf (H_\psi^k)}$  or ${\bf (C_2)}$, ${\bf (H_B^i)}$, ${\bf (H_\chi^j)}$ and ${\bf (H_\psi^k)}$ hold with $i,j,k\in\{1,2\}$.
Then the inequality system $(S)$ possesses at least one solution.
\end{theorem}

\begin{proof}
For each $R>0$ consider $(S_R)$ to be the system of inequalities obtained from $(S)$ but with $K_R:=K\cap \bar{B}_X(0,R)$ and $\Lambda_R:=\Lambda\cap \bar{B}_Y(0,R)$ instead of $K$ and $\Lambda$, respectively.  Then there exists at least one solution $[u_R,\lambda_R]\in K_R\times \Lambda_R$ for $(S_R)$. 

We claim that, regardless whether ${\bf (C_1)}$ or ${\bf (C_2)}$ holds, there exists $R_0>0$ such that the corresponding solution $[u_{R_0},\lambda_{R_0}]$ satisfies $\max\{\|u_{R_0}\|_X,\|\lambda_{R_0}\|_{Y}\}<R_0$. 

Arguing by contradiction, assume that for any $R>0$ and any solution $[u_R,\lambda_R]$ of $(S_R)$ one has $\max\{ \|u_R\|_X,\|\lambda_R\|_Y\}=R$. Choosing $v:=0_X$ and $\mu:=0_Y$ in $(S_R)$, then adding the two inequalities  we get 
\begin{equation}\label{CoercivityEst1}
B(0_X, \lambda_{R})-B(u_{R},0_Y)+\chi(u_{R_0},-u_{R})+\psi( \lambda_{R},- \lambda_{R})\geq -\langle f,u_{R}\rangle -\langle g, \lambda_{R} \rangle.
\end{equation}

\begin{description}
\item {\sc Case 1.}  ${\bf {\bf (C_1)}}$, ${\bf (H_B^1)}$, ${\bf (H_\chi^j)}$ and ${\bf (H_\psi^k)}$ hold.

Then
$$
\chi(u_{R},-u_{R})+\psi( \lambda_{R},- \lambda_{R})\geq -\|f\|_{X^\ast}\|u_{R}\|_{X}-\|g\|_{Y^\ast}\|\lambda_R\|_{Y}+ B(u_R,0_Y)-B(0_X,\lambda_R).
$$
Since every convex lower continuous functional is bounded below by an affine function (see, e.g., Brezis \cite[Proposition 1.10]{Brezis2011}), there exist $f_1\in X^\ast$, $g_1\in Y^\ast$ and $c_1,c_2\in\mathbb{R}$ such that 
$$
B(u_R,0_Y)\geq \langle f_1,u_R\rangle+c_1\geq -\|f_1\|_{X^\ast}\|u_R\|_X+c_1, 
$$
and 
$$
-B(0_X,\lambda_R)\geq  \langle g_1,\lambda_R\rangle+c_2\geq -\|g_1\|_{Y^\ast}\|\lambda_R\|_Y+c_2.
$$
Thus,
$$
\frac{\chi(u_R,-u_R)+\psi(\lambda_R,-\lambda_R)}{\sqrt{ \|u_R\|_X^2+\|\lambda_R\|_Y^2}}\geq \frac{-c_3\|u_R\|_X-c_4\|\lambda_R\|_Y+c_1+c_2}{\sqrt{ \|u_R\|_X^2+\|\lambda_R\|_Y^2}},
$$
where $c_3:=\|f\|_{X^\ast}+\|f_1\|_{X^\ast}>0$ and $c_4:=\|g\|_{Y^\ast}+\|g_1\|_{Y^\ast}>0$.
Letting $R\to\infty$ we reach a contradiction as the left-hand side term tends to $-\infty$, while the right-hand side term is bounded (keep in mind we assumed $\max\{\|u_R\|_X,\|\lambda_R\|_Y\}=R$).

\medskip

\item {\sc Case 2.} ${\bf {\bf (C_2)}}$, ${\bf (H_B^i)}$,  ${\bf (H_\chi^j)}$ and ${\bf (H_\psi^k)}$ hold.

Using  \eqref{CoercivityEst1} we get
$$
\frac{B(0_X,\lambda_R)-B(u_R,0_Y)}{\sqrt{\|u_R\|_X^2+\| \lambda_R \|_Y^2 }^{\max\{p,q\}}}\geq 
-\frac{\|f\|_{X^\ast} \|u_R\|_X+m_\chi\|u_R\|_X^p+\|g\|_{Y^\ast} \|\lambda_R\|_Y+m_\psi\|\lambda_R\|_X^q}{ \sqrt{\|u\|_X^2+\|\lambda\|_Y^2 }^{\max\{p,q\}} },
$$
 and a contradiction is reached by letting  $R\to\infty$. 
\end{description}

\medskip

Now, let $[u_{R_0},\lambda_{R_0}]$ be the solution of $(S_{R_0})$ such that $\|u_{R_0}\|_{X}<R_0$ and $\|\lambda_{R_0}\|_{Y}<R_0$. Then $[u_{R_0},\lambda_{R_0}]$ also solves $(S)$. In order to prove this let $[v,\mu]\in K\times\Lambda$ be fixed. Then the number 
$$
t:=\left\{
\begin{array}{ll}
\frac{1}{2}, & \mbox{ if } v=u_{R_0} \mbox{ and }\mu=\lambda_{R_0}\\
\min\left\{\frac{1}{2},\frac{R_0-\|\lambda_{R_0}\|_{Y}}{\|\mu-\lambda_{R_0}\|_Y}\right\}, & \mbox{ if }v=u_{R_0}\mbox{ and }\mu\neq\lambda_{R_0} \\
\min\left\{\frac{1}{2},\frac{R_0-\|u_{R_0}\|_{X}}{\|v -u_{R_0}\|_X}\right\}, & \mbox{ if } v\neq u_{R_0}\mbox{ and }\mu=\lambda_{R_0}\\
\min\left\{\frac{1}{2},\frac{R_0-\|u_{R_0}\|_X}{\|v-u_{R_0}\|_X},\frac{R_0-\|\lambda_{R_0}\|_{Y}}{\|\mu-\lambda_{R_0}\|_Y}\right\}, &\mbox{ if } v\neq u_{R_0}\mbox{ and }\mu\neq\lambda_{R_0}\\
\end{array}
\right.
$$
belongs to $(0,1)$ and $[v_t,\mu_t]:=\left[u_{R_0}+t(v-u_{R_0}),\lambda_{R_0}+t(\mu-\lambda_{R_0})\right]\in K_{R_0}\times \Lambda_{R_0}$. 

\medskip

If ${\bf (H_B^1)}$ holds, then the convexity of the functionals $u\mapsto B(u,\lambda)$ and $\lambda\mapsto -B(u,\lambda)$ ensures that 
\begin{equation}\label{ConvexityEst1}
B(v_t,\lambda_{R_0})-B(u_{R_0},\lambda_{R_0})\leq t\left[ B(v,\lambda_{R_0})-B(u_{R_0},\lambda_{R_0}) \right],
\end{equation}
and 
\begin{equation}\label{ConvexityEst2}
B(u_{R_0},\lambda_{R_0})-B(u_{R_0},\mu_t)\leq t\left[ B(u_{R_0},\lambda_{R_0})-B(u_{R_0},\mu) \right],
\end{equation}
while if ${\bf (H_B^2)}$ holds, \eqref{ConvexityEst1}-\eqref{ConvexityEst2} follow directly from \eqref{Concavity1} and \eqref{Concavity2}.

We only consider the case when ${\bf (H_\chi^1)}$, ${\bf (H_\psi^2)}$ are fulfilled, the others being similar. Since $[u_{R_0},\lambda_{R_0}]$ solves $(S_{R_0})$ and \eqref{ConvexityEst1}-\eqref{ConvexityEst2} are holding we have 
\begin{eqnarray*}
t\langle f,v-u_{R_0} \rangle &=& \langle f,v_t-u_{R_0}\rangle\leq  B(v_t,\lambda_{R_0})-B(u_{R_0},\lambda_{R_0})+\chi(u_{R_0},t(v-u_{R_0}))\\
	& \leq & t\left[ B(v,\lambda_{R_0})-B(u_{R_0},\lambda_{R_0})+\chi(u_{R_0},v-u_{R_0})\right]+(1-t)\underbrace{\chi(u_{R_0},0_X)}_{=0},
\end{eqnarray*}
and
\begin{eqnarray*}
t\langle g,\mu-\lambda_{R_0} \rangle &=& \langle g,\mu_t-\lambda_{R_0}\rangle\leq  B(u_{R_0},\lambda_{R_0})-B(u_{R_0},\mu_t)+\psi(\lambda_{R_0},t(\mu-\lambda_{R_0}))\\
	& \leq & t\left[B(u_{R_0},\lambda_{R_0})-B(u_{R_0},\mu)+\psi(u_{R_0},v-u_{R_0})\right].
\end{eqnarray*}
Dividing both inequalities by $t>0$ we infer that $[u_{R_0},\lambda_{R_0}]$ indeed solves $(S)$ as $[v,\mu]\in\mathcal{K}$ was chosen arbitrarily.
\qed
\end{proof}

\section{Applications}

\subsection{Partial differential inclusions driven by the $\Phi$-Laplacian}

Let $\Omega$ be a bounded connected  open subset of $\mathbb{R}^N$, with Lipschitz boundary $\Gamma$ partitioned into three measurable parts $\Gamma_1, \Gamma_2$ and $\Gamma_3$ such that ${\rm meas}(\Gamma_i)>0$, $i\in\{1,2,3\}$. We consider the following boundary problem: 
$$
(P):\ \left\{
\begin{array}{ll}
\Delta_\Phi u\in \partial_2 h(x,u(x)),& \mbox{ in } \Omega,\\
u=0,& \mbox{ on }\Gamma_1,\\
\frac{\partial u}{\partial n_\Phi}=f_2,& \mbox{ on }\Gamma_2,\\
\left| \frac{\partial u}{\partial n_\Phi}\right|\leq g, \; \frac{\partial u}{\partial n_\Phi}=-g\frac{u}{|u|}\mbox{ if } u\neq 0, & \mbox{ on }\Gamma_3,
\end{array}
\right.
$$
where $\Delta_\Phi u:={\rm div}(\frac{\phi(|\nabla u|)}{|\nabla u|}\nabla u)$ is the $\Phi$-Laplace operator, with $\Phi$ being the $N$-function (i.e., $\Phi$ is convex and even, $\Phi(t)=0\Leftrightarrow t=0$, $\lim_{t\to 0} \frac{\Phi(t)}{t}=0$ and $\lim_{t\to\infty}\frac{\Phi(t)}{t}=\infty$) defined by
$$
\Phi(t):=\int_0^t \phi(s)\; ds.
$$ 
Here and hereafter $\partial_2 h(x,t)$ stands for the subdifferential (in the sense of Convex Analysis) of $t\mapsto h(x,t)$ and  $\frac{\partial u}{\partial n_\Phi}:=\frac{\phi(|\nabla u|)}{|\nabla u|}\nabla u \cdot n$, with $n$ being the unit outer normal vector to $\Gamma$.

In the sequel we always assume that   
\begin{enumerate}[$(\mathcal{H}_0)$] 
\item $\phi:\mathbb{R}\to\mathbb{R}$ is continuous, odd, strictly increasing and onto such that 
 $$1<\phi^-\leq \phi^+<\infty,$$
where $\phi^-:=\inf_{t>0} \frac{t\phi(t)}{\Phi(t)}$ and $\phi^+:=\sup_{t>0} \frac{t\phi(t)}{\Phi(t)}$;
\end{enumerate}

Note that, if $\phi$ satisfies $(\mathcal{H}_0)$, then $\phi^{-1}$ also satisfies $(\mathcal{H}_0)$ and the following relations hold (see, e.g., \cite[Lemma C.6]{C-P-S-T}) 
$$
\frac{1}{(\phi^{-1})^-}+\frac{1}{\phi^+}=1=\frac{1}{(\phi^{-1})^+}+\frac{1}{\phi^+}.
$$
We point out the fact that if $(\mathcal{H}_0)$ holds, then $\Phi$ satisfies the $\Delta_2$-condition for large numbers, i.e., 
$$
\Phi(2t)\leq k\Phi(t), \ \forall t\geq t_0,
$$
for some positive constants $k$ and $t_0$.

Due to the presence of the $\Phi$-Laplacian the suitable function space to seek weak solutions of Problem $(P)$ is the {\it Orlicz-Sobolev} space $W^{1,\Phi}(\Omega)$. We recall below the definition and some basic properties of the Orlicz and Orlicz-Sobolev spaces that will be used to derive a variational formulation for our problem, then to prove the existence of at least one weak solution. For more details we refer to \cite{Adams,C-P-S-T,GH-L-M-S,Kras-Rut}.

The {\it Orlicz} space $L^\Phi(\Omega)$ is defined by 
$$
L^\Phi(\Omega):=\left\{ u:\Omega\to \mathbb{R} \mbox{ measurable}:  \ \int_\Omega \Phi(|u|)\; dx<\infty\right\},
$$
and endowed with the {\it Luxemburg} norm 
$$
|u|_\Phi:=\inf\left\{ k>0:\ \int_\Omega \Phi\left(\frac{|u|}{k}\right)dx\leq 1 \right\}
$$
becomes a separable and reflexive Banach space. 

The {\it Orlicz-Sobolev} space $W^{1,\Phi}(\Omega)$ is defined by 
$$
W^{1,\Phi}(\Omega):=\left\{ u\in L^\Phi(\Omega):\ |\nabla u|\in L^\Phi(\Omega) \right\}
$$
and it is endowed with the norm 
$$
\|u\|_{1,\Phi}:=|u|_\Phi+|\nabla u|_\Phi.
$$
Using the $\Delta_2$-condition we can identify $(L^{\Phi}(\Omega))^\ast$ with $L^{\Phi^\ast}(\Omega)$, where $\Phi^\ast$ is the {\it complementary} function of $\Phi$, i.e., 
$$
\Phi^\ast(s):=\sup_{t\geq 0}\{st-\Phi(t)\}, \ s\geq 0.
$$
Hypothesis $(\mathcal{H}_0)$ ensures that $\Phi^\ast$ is also an $N$-function and satisfies the $\Delta_2$-condition for large numbers as
$$
\Phi^\ast(s)=\int_0^s \phi^{-1}(t)\; dt.
$$
Moreover, the following  H\" older-type inequality 
$$
\int_\Omega uv \; dx\leq 2|u|_\Phi|v|_{\Phi^\ast}, 
$$ 
holds for any $u\in L^\Phi(\Omega), v\in L^{\Phi^\ast}(\Omega)$.

For   two $N$-functions $\Phi,\Psi$  if there exist $k,t_0>0$ such that 
$$
\Phi(t)\leq \Psi(kt), \ \forall t\geq t_0,
$$
then we say that $\Psi$ {\it dominates $\Phi$ near infinity}  and write $\Psi \succ \Phi$. Note that if $\Psi\succ \Phi$, then the embedding $L^\Psi(\Omega)\hookrightarrow L^\Phi(\Omega)$ is continuous. Two $N$-functions are called {\it equivalent near infinity} if they dominate each other near infinity.  We say that $\Psi$ {\it grows essentially faster than $\Phi$ near infinity}, and we write $\Psi\succ\succ\Phi$, if $\lim_{t\to\infty}\frac{\Phi(t)}{\Psi(kt)}=0$, for all $k>0$. Note that if  $\Psi\succ\succ\Phi$, then $\Psi\succ\Phi$.
Henceforth we always assume that 
$$
\int_0^1 \frac{\Phi^{-1}(s)}{s^{\frac{N+1}{N}}}\; ds<\infty,
$$
otherwise we replace $\Phi$ with an equivalent $N$-function near infinity. If in addition 
\begin{equation}\label{DivInt}
\int_1^\infty \frac{\Phi^{-1}(s)}{s^{\frac{N+1}{N}}}\; ds=\infty,
\end{equation}
then we define the {\it Sobolev conjugate} function of $\Phi$, denoted $\Phi_\ast$, to be the inverse of the mapping  
$$
t\mapsto \int_0^t \frac{\Phi^{-1}(s)}{s^{\frac{N+1}{N}}}\; ds.
$$ 
This function plays the same role for Orlicz-Sobolev spaces as the critical exponent $p^\ast$ in the case of classical Sobolev spaces. More precisely, if \eqref{DivInt} holds and $\Phi_\ast\succ\succ\Psi$, then the embedding $W^{1,\Phi}(\Omega)\hookrightarrow L^\Psi(\Omega)$  is compact and the embedding $W^{1,\Phi}(\Omega)\hookrightarrow L^{\Phi_\ast}(\Omega)$ is continuous. If \eqref{DivInt} is not satisfied, then the embedding $W^{1,\Phi}(\Omega)\hookrightarrow L^\Psi(\Omega)$ is compact for any $N$-function $\Psi$. It is well-known that $\Phi_\ast\succ\succ \Phi$, whenever \eqref{DivInt} holds.

Since $\Phi$ and $\Phi^\ast$ satisfy the $\Delta_2$-condition for large numbers holds, it follows that $C^\infty(\bar{\Omega})$ is dense in $W^{1,\Phi}(\Omega)$ and the the {\it trace  operator} $\gamma: W^{1,\Phi}(\Omega)\to L^\Phi(\Gamma)$, defined by  $\gamma u:=u|_{\Gamma}$ for all $u\in C^\infty(\bar{\Omega})$ is continuous. By \cite[Theorem 3.8 and Corollary 3.3]{Don-Tru}, if \eqref{DivInt}  holds and $ (\Phi_\ast)^{(N-1)/N}\succ \succ\Phi$, then the trace operator is also compact, i.e., the embedding $W^{1,\Phi}(\Omega) \hookrightarrow L^\Phi(\Gamma)$ is compact.

Next, we derive a variational formulation via Lagrange multipliers for problem $(P)$ and show the existence of at least one weak solution provided the following conditions are fulfilled.

\begin{enumerate}[$(\mathcal{H}_1)$]
\item $h:\Omega \times \mathbb{R}\to\mathbb{R} $ is Carath\' eodory function such that 

\begin{enumerate}[$(i)$]

\item $t\mapsto h(x,t)$ is convex for a.e. $x\in\Omega$;

\item $h(x,u(x))\in L^1(\Omega)$ for all $u\in W^{1,\Phi}(\Omega)$;

\end{enumerate}

\item  $f_2\in L^{\Phi^\ast}(\Gamma_2)$, $g\in L^\infty(\Gamma_3)$, $g\geq 0$ a.e. on $\Gamma_3$.

\end{enumerate}

Assume that $u$ is a sufficiently smooth solution of problem $(P)$.  Multiplying the first line of $(P)$ by $v-u$, then integrating over $\Omega$ we get 
\begin{equation}\label{VarForm1}
-\int_\Omega \frac{\phi(|\nabla u|)}{|\nabla u|}\nabla u\cdot \nabla(v-u) dx+\int_{\Gamma} \frac{\partial u}{\partial n_\Phi}(v-u)\; d\sigma=\int_{\Omega} \xi(v-u)\; dx;
\end{equation} 
for some $\xi(x)\in \partial_2 h(x,u(x))$. The definition of the convex subdifferential implies that 
\begin{equation}\label{VarForm2}
\int_\Omega \xi (v-u)\; dx\leq \int_\Omega \left[h(x,v(x))-h(x,u(x))\right]dx.
\end{equation}
 In order to deal with the first boundary condition we consider the following constraint set of {\it admissible displacements}
$$
X:=\left\{u\in W^{1,\Phi}(\Omega):\ \gamma u=0 \mbox{ on }\Gamma_1  \right\},
$$
Here and hereafter we simply write $u$ instead of $\gamma u$ to denote the trace of $u$ on $\Gamma$. It is readily seen that $X$ is a closed subspace of $W^{1,\Phi}(\Omega)$.  Here and hereafter we endow $X$ with the norm 
$$
\|u\|:=|\nabla u|_\Phi,
$$
which is equivalent to the norm $\|\cdot \|_{1,\Phi}$ inherited  from $W^{1,\Phi}(\Omega)$. This is a simple consequence of the following theorem which is probably known, but we provide the proof for the sake of completeness.

\begin{theorem}\label{EquivNorms}
Let $\Omega\subset \mathbb{R}^N$ be a bounded domain with Lipschitz boundary $\Gamma$ and let $\Gamma_1\subset \Gamma$ be such that ${\rm meas}(\Gamma_1)>0$.  If \eqref{DivInt} holds, then assume in addition that  $(\Phi_\ast)^{\frac{N-1}{N}}\succ\succ \Phi$. Then 
$$
\|u\|:=|\nabla u|_\Phi+\int_{\Gamma_1} |u|\; d\sigma
$$ 
is a norm on $W^{1,\Phi}(\Omega)$ and it is equivalent to the usual norm $\|u\|_{1,\Phi}:=|\nabla u|_{\Phi}+|u|_\Phi$.
\end{theorem} 

\begin{proof}
One can easily check that $\|\cdot\|$ is a seminorm on $W^{1,\Phi}(\Omega)$, i.e., it is positive homogeneous and subadditive.  Assume now that $\|u\|=0$.  Then 
$$
|\nabla u|_{\Phi}=0 \mbox{ and }\int_{\Gamma_1} |u|\; d\sigma=0.
$$
The fact that $|\nabla u|_\Phi=0$ implies that $u$ is constant in $\Omega$, say $u(x)=k$, while from the second equality we have 
$$
0=\int_{\Gamma_1} |k|\; d\sigma=|k|{\rm meas}(\Gamma_1),
$$
which forces $k=0$, therefore $\|u\|=0\Rightarrow u=0$. 

We show next there exists $m>0$ such that 
\begin{equation}\label{FirstIneqEquivNorms}
m\|u\|_{1,\Phi}\leq \|u\|, \ \forall u\in W^{1,\Phi}(\Omega).
\end{equation}
Arguing by contradiction, assume that for each $n\geq 1$ there exists $u_n\in W^{1,\Phi}(\Omega)$ such that 
$$
\frac{1}{n}\|u_n\|_{1,\Phi}>\|u_n\|.
$$
In particular, $\|u_n\|_{1,\Phi}\neq 0$, hence the sequence  $v_n:=\frac{u_n}{\|u_n\|_{1,\Phi}}$ satisfies 
\begin{equation}\label{SecondIneqEquivNorms}
\|v_n\|_{1,\Phi}=1 \mbox{ and }\|v_n\|<\frac{1}{n}, \ \forall n\geq 1.
\end{equation}
Then 
$$
|\nabla v_n|_{\Phi}\to 0 \mbox{ and }\int_{\Gamma_1} |v_n|\; d\sigma\rightarrow 0, \mbox{ as }n\to\infty.
$$
On the other hand, $\{v_n\}$ is a bounded sequence in the reflexive Banach space $W^{1,\Phi}(\Omega)$, therefore there exists a subsequence (for simplicity we do not relabel) and $v\in W^{1,\Phi}(\Omega)$ such that 
$$
v_n\rightharpoonup v \mbox{ in }W^{1,\Phi}(\Omega).  
$$
Since $W^{1,\Phi}(\Omega)$ is compactly embedded into $L^{\Phi}(\Omega)$ and $L^{\Phi}(\Gamma_1)$ it follows that 
$$
v_n\to v \mbox{ in }L^{\Phi}(\Omega) \mbox{ and }v_n\to v \mbox{ in }L^{\Phi}(\Gamma_1).
$$
Let $w\in C_0^\infty(\Omega)$ be a test function and $i\in\{1,\ldots,N\}$ a fixed index. Then 
$$
\int_\Omega v_n(x)\frac{\partial w}{\partial x_i}(x)\; dx=-\int_\Omega \frac{\partial v_n}{\partial x_i}(x)w(x)\;dx\to 0  \mbox{ as }n\to \infty
$$
and 
$$
\int_\Omega v_n(x)\frac{\partial w}{\partial x_i}(x)\; dx\to \int_\Omega v(x)\frac{\partial w}{\partial x_i}(x)\; dx, \mbox{ as }n\to \infty,
$$
which leads to 
$$
0=\int_{\Omega} v(x)\frac{\partial w}{\partial x_i}(x)\; dx=-\int_{\Omega} \frac{\partial v}{\partial x_i}(x)w(x)\; dx.
$$
Since $w$ and $i$ were arbitrarily fixed last relation shows that $|\nabla v|=0$ a.e. in $\Omega$. In particular, 
$v_n$ converges strongly to $v$ in $W^{1,\Phi}(\Omega)$ as
$$
\|v_n-v\|_{1,\Phi}=|\nabla v_n|_{\Phi}+|v_n-v|_{\Phi} \to 0, \mbox{ as }n\to \infty.
$$
One also has
$$
\left| \int_{\Gamma_1} |v_n|\; d\sigma-\int_{\Gamma_1}|v|\; d\sigma\right|\leq \int_{\Gamma_1} |v_n-v|\; d\sigma\leq 2{\rm meas}(\Gamma_1) |v_n-v|_{\Phi}\to 0,
$$
which shows that ${\displaystyle \int_{\Gamma_1}|v|\; d\sigma=0}$. It follows that $v=0$ and this is a contradiction since $v_n\to v$ in $W^{1,\Phi}(\Omega)$ and $\|v_n\|_{1,\Phi}\to 1$. 
 
The equivalence of $\|\cdot\|$ and $\|\cdot\|_{1,\Phi}$ follows now from 
$$
\|u\|=|\nabla u|_\Phi+\int_{\Gamma_1}|u|\; d\sigma\leq \|u\|_{1,\Phi}+2{\rm meas}(\Gamma_1)|u|_{L^{\Phi}(\Gamma_1)}\leq (1+2{\rm meas}(\Gamma_1)c_1)\|u\|_{1,\Phi},
$$
with $c_1>0$ the constant given by the compact embedding $W^{1,\Phi}(\Omega)\hookrightarrow L^{\Phi}(\Gamma_1)$.
 \qed
\end{proof}

Defining the Lagrange multiplier $\lambda\in X^\ast$ by 
$$
\langle \lambda,v\rangle:=\int_{\Gamma_3} -\frac{\partial u}{\partial n_\Phi} v\; d\sigma,
$$ 
and the set of {\it admissible Lagrange multipliers} 
$$
\Lambda:=\left\{ \mu\in X^\ast:\ \langle \mu,v\rangle\leq \int_{\Gamma_3} g|v|\; d\sigma, \ \forall v\in X \right\}.
$$
one can easily check that $\lambda\in \Lambda$. Moreover, for every $\mu\in\Lambda$  one has
\begin{equation}\label{VarForm3}
\langle \lambda,u\rangle-\langle \mu,u\rangle=\int_{\Gamma_3} -\frac{\partial u}{\partial n_\Phi}u d\sigma-\langle \mu,u\rangle =\int_{\Gamma_3} g\frac{u}{|u|}ud\sigma-\langle \mu,u\rangle =\int_{\Gamma_3} g|u|d\sigma-\langle \mu,u\rangle \geq 0.
\end{equation}
Keeping in mind \eqref{VarForm1}, \eqref{VarForm2} and \eqref{VarForm3} we get the following variational formulation of  $(P)$. 

\noindent $(\mathcal{P}_V)$: Find $u\in X$  and $\lambda\in\Lambda$ such that 
$$
\left\{
\begin{array}{l}
{\displaystyle \int_\Omega \frac{\phi(|\nabla u|)}{|\nabla u|}\nabla u\cdot\nabla(v-u)\;dx+\int_\Omega [h(x,v)-h(x,u)]dx+\langle\lambda,v-u\rangle\geq \int_{\Gamma_2} f_2 (v-u)\; d\sigma}, \forall v\in X,\\
\langle \lambda-\mu,u\rangle\geq 0, \forall \mu\in\Lambda.
\end{array}
\right.
$$
If $[u,\lambda]$ solves $(\mathcal{P}_V)$ we say that $u$ is a {\it weak solution} of $(P)$ with the {\it corresponding Lagrange multiplier} $\lambda$.

\begin{theorem}\label{MainRes2}
Suppose $(\mathcal{H}_0)$, $(\mathcal{H}_1)$  and $(\mathcal{H}_2)$ hold. If  \eqref{DivInt} holds, then assume in addition that $(\Phi_\ast)^{\frac{N-1}{N}}\succ\succ \Phi$. Then $(\mathcal{P}_V)$ possesses at least one solution.
\end{theorem}

\begin{proof}
Let $K:=X$, $Y:=X^\ast$ and define $B:X\times Y\to\mathbb{R}$, $\chi:X\times X \to \mathbb{R}$, $\psi:Y\times Y\to\mathbb{R}$, $f\in X^\ast$ and $g\in X$ by 
$$
B(u,\lambda):=\langle \lambda,u\rangle+\int_\Omega h(x,u(x))dx, \ 
\chi(u,v):=\langle I'(u),v\rangle,$$
$$
 \psi(\lambda,\mu):=0,\ \langle f,v\rangle :=\int_{\Gamma_2} f_2 v\; d\sigma, \ g=0_X,
$$
where $I:X\to\mathbb{R}$ is the convex and lower semicontinuous functional defined by 
$$
I(u):=\int_\Omega \Phi(|\nabla u |)\;dx.
$$
The functional $I\in C^1(X;\mathbb{R})$  and one has (see, e.g., \cite[Lemma 3.4]{GH-L-M-S})
$$
\langle I'(u),v\rangle= \int_\Omega\frac{\phi(|\nabla u|)}{|\nabla u|}\nabla u\cdot\nabla v\; dx.
$$
It is straightforward that $(\mathcal{P}_V)$ can be written as system $(S)$  with   $B,\chi,\psi, K,\Lambda, f\in X^\ast$ and  $g\in X$ as above and conditions  ${\bf (H_B^1)}$, ${\bf (H_\chi^2)}$, ${\bf (H_\psi^1)}$ are fulfilled.  In order to complete the proof it suffices to show that the coercivity condition ${\bf (C_1)}$ holds. We start by pointing out that $\Lambda$ is a bounded subset of $X^\ast$ as 
$$
\langle  \mu,v\rangle \leq \int_{\Gamma_3} g|v|\; d\sigma\leq 2\|g\|_{L^\infty(\Gamma_3)}{\rm meas}(\Gamma_3)|v|_{L^{\Phi}(\Gamma_3)}\leq c_0\|v\|, \forall v\in X, \mu\in \Lambda,
$$
for some suitable constant $c_0>0$. This shows that $\|\mu\|_\ast\leq c$ for all $\mu\in\Lambda$. In particular, for $[u,\lambda]\in X\times \Lambda$ one has $\sqrt{\|u\|^2+\|\lambda\|_\ast^2}\to \infty$ if and only if $\|u\|\to \infty$. 
On the other hand, one has  (see, e.g., \cite[Lemma C.9]{C-P-S-T})
$$
I(u)\geq \|u\|^{\phi^-}, \forall u\in X, \|u\|>1,
$$
Thus, 
\begin{eqnarray*}
\frac{\chi(u,-u)+\psi(\lambda,-\lambda)}{\sqrt{\|u\|^2+\|\lambda\|_\ast^2}}&=&-\frac{\langle I'(u),u\rangle }{\sqrt{\|u\|^2+\|\lambda\|_\ast^2}} = -\frac{\displaystyle \int_\Omega \phi(|\nabla u|)|\nabla u|\; dx}{\sqrt{\|u\|^2+\|\lambda\|_\ast^2}} 
\leq -\frac{\phi^-\displaystyle \int_\Omega \Phi(|\nabla u|)\;dx }{\sqrt{\|u\|^2+\|\lambda\|_\ast^2}}\\
&\leq & -\frac{\phi^-\|u\|^{\phi^-}}{\sqrt{\|u\|^2+c_0^2}}\to -\infty \mbox{ as }\|u\|\to \infty.
\end{eqnarray*}
In conclusion $(\mathcal{P}_V)$ possesses at least one solution due to Theorem \ref{Unbounded1}.
\qed
\end{proof}

We close this subsection with some comments on the choice of problem $(P)$.  Very recently (see \cite[Section 6]{Cos-Pit21}) the following problem was investigated:
$$
(P_1): \ \left\{
\begin{array}{ll}
-\Delta_\Phi u+\frac{\phi(|u|)}{|u|}u =f_1+f_2, & \mbox { in } \Omega\\
-f_2\in \partial_C^2 h(x,u), & \mbox{ in }\Omega\\
u=0 ,& \mbox{ on } \Gamma_1\\ 
-\frac{\partial u}{\partial n_\Phi}\in k(x,u)\partial_C^2 j(x,u), & \mbox{ on }\Gamma_2
\end{array}
\right.
$$
where $h:\Omega\times \mathbb{R}\to\mathbb{R}$ and $j:\Gamma_2\times\mathbb{R}\to \mathbb{R}$ are locally Lipschitz w.r.t. the second variable and $\partial_C^2 h(x,t)$ stands for the {\it Clarke subdifferential} (see, e.g., Clarke \cite{Clarke}) of the mapping $t\mapsto g(x,t)$. The fact that we opted to partition $\Gamma$ into three parts for problem $(P)$, although it poses  no mathematical difficulties, comes from the fact that $(P)$ can serve as a model for the antiplane shear deformation of cylinders, of Hencky-type material, in contact with a rigid foundation as it can be seen in the following subsection. 
Due to the fact that convex functions are in fact locally Lipschitz on the interior of their effective domain and the convex and Clarke subdifferentials coincide in this case, problem $(P_1)$ is in fact a bit more general than $(P)$. The reasoning for  choosing the convex sudifferential here is the following: in \cite{Cos-Pit21} the authors show (through a different approach)  that $(P_1)$ possesses at least one weak solution under the key assumption that the Clarke subdifferential of $h$ satisfies a growth condition of the following type 
\begin{equation}\label{Integrability}
|\partial_C^2 h(x,t)|\leq a_1(x)+c_1\psi(|t|), \ \mbox{ for a.e. }x\in\Omega \mbox{ and all }t\in\mathbb{R},
\end{equation}
where $a_1\in L^{\Phi^\ast}(\Omega)$ and $\psi:\mathbb{R}\to\mathbb{R}$ satisfies $(\mathcal{H}_0)$, $\Phi_\ast\succ\succ \Psi$ and  $\psi^+<\phi^-$. This ensures, on the one hand, the integrability of $x\mapsto h^0(x,u(x);v(x))$ for all $u,v\in W^{1,\Phi}(\Omega)$ and, on the other hand, the inequality $\psi^+<\phi^-$ ensures the coercivity of a certain set-valued mapping. Choosing the convex subdifferential for problem $(P)$ we are able to prove the existence of at least one weak solution under the considerably less restrictive condition that $h(x,u(x))\in L^1(\Omega)$ for any $u\in W^{1,\Phi}(\Omega)$. Note that \eqref{Integrability} ensures this (if we replace $\partial_C^2$ by $\partial_2$), but we do not need to impose $\psi^+<\phi^-$, but only the fact that $\Phi_\ast\succ\succ \Psi$, if \eqref{DivInt} holds.  For example, if $\phi(t):=|t|^{p-2}t$ and $\psi(t):=|t|^{q-2}t$, $p,q\in (0,\infty)$, then $\phi^-=\phi^+=p$ and $\psi^-=\psi^+=q$, respectively. Therefore, $(P_1)$ possesses at least one weak solution if $q\in(1,p)$, whereas $(P)$ has has a solution if $q\in (1,p^\ast)$, with 
$$
p^\ast:=\left\{
\begin{array}{ll}
\frac{Np}{N-p}, & \mbox{ if }p<N,\\
\infty, & \mbox{ otherwise.}
\end{array}
\right.
$$ 

\subsection{An example arising in Contact Mechanics}

Throughout this subsection we consider a mathematical model which describes the frictional contact between a nonlinear elastic body and a rigid foundation. We investigate the {\it antiplane shear deformation} of the body, i.e., the deformation expected by loading a long cylinder in the direction of its generators so that the displacement field is independent of the axial coordinate. The antiplane model is in Contact Mechanics due to the fact that it maintains physical relevance while significantly simplifying the equations. For more details and connections we refer to \cite{And-Cos-Mat,Cos-Kri-Var,Cos-Mat10,HS02,Hor95,M-O-S,Mat-Sof09}.

Let  ${\bf B}$  be a deformable cylinder in the cartesian system $Ox_1x_2x_3$. We assume ${\bf B}$  is made of a nonlinear elastic Hencky-type material and its generators are parallel to the $Ox_3$-axis and are long enough such that the end effects in the axial direction are negligible. The cross section is a bounded domain $\Omega$ in the plane $Ox_1x_2$, with Lipschitz boundary $\Gamma$, partitioned into three measurable parts of positive measure $\Gamma_i$,  $1\leq i\leq3$.  Without loss of generality we may assume ${\bf B}=\Omega\times\mathbb{R}$. We assume ${\bf B}$ is subjected to volume forces of density $\overrightarrow{F_0}$ and surface tractions of density $\overrightarrow{F_2}$ act on $\Gamma_2\times \mathbb{R}$. Furthemore, suppose the body is clamped on $\Gamma_1\times\mathbb{R}$ and in frictional contact with a rigid foundation on $\Gamma_3\times\mathbb{R}$.

We define $\mathbb{S}_3$ to be the linear space of symmetric tensors of second order in $\mathbb{R}^3$. In order to avoid confusion we adopt the following notations: tensors in $\mathbb{S}_3$ will be bolded and vectors in $\mathbb{R}^d$ ($d=2,3$) will be written with an arrow above. For a vector $\overrightarrow{u}$ we denote by $u_\nu:=\overrightarrow{u}\cdot \overrightarrow {\nu}$ its the normal component and by $\overrightarrow{u_\tau}:=\overrightarrow{u}-u_\nu\overrightarrow \nu$ its tangential component. Similarly, for a stress field $\boldsymbol{\sigma}$ we define $\sigma_\nu$ and $\overrightarrow {\sigma_\tau}$ to be the normal and the tangential components of the Cauchy vector $\boldsymbol{\sigma}\overrightarrow{\nu}$, i.e., $\sigma_\nu:=(\boldsymbol{\sigma}\overrightarrow{\nu})\cdot\overrightarrow{\nu}$ and $\overrightarrow{\sigma_\tau}:=\boldsymbol{\sigma}\overrightarrow{\nu}-\sigma_\nu\overrightarrow{\nu}$, respectively. We also consider
$$
{\rm Div\; }\boldsymbol{\sigma}:=(\mu_1,\mu_2,\mu_3),\ \mu_i:=\frac{\partial \sigma_{i1}}{\partial x_1}+\frac{\partial \sigma_{i2}}{\partial x_2}+\frac{\partial \sigma_{i3}}{\partial x_3}, 
$$
$$
\boldsymbol{\varepsilon}(\overrightarrow{u}):=(\varepsilon_{ij}(\overrightarrow{u}))_{1\leq i,j\leq 3}, \ \varepsilon_{ij}(\overrightarrow{u}):=\frac{1}{2}\left( \frac{\partial u_i}{\partial x_j}+\frac{\partial u_j}{\partial x_i}\right),
$$
and 
$$
\boldsymbol{\sigma}^D:=\boldsymbol{\sigma}-\frac{1}{3} tr(\boldsymbol{\sigma})I_3, \ \ tr(\boldsymbol{\sigma}):=\sigma_{11}+\sigma_{22}+\sigma_{33} \mbox{ and } I_3=\left(
\begin{array}{lll}
1 \ \ & 0\ \ & 0\\
0 & 1 & 0\\
0 & 0 & 1
\end{array}
\right).
 $$

The mathematical model which describes the contact between the cylindrical body ${\bf B}$ and the foundation is presented below.

\noindent $(P_{3D}):$ Find a displacement field  $\overrightarrow{u}:{\Omega\times\mathbb{R}}\to\mathbb{R}^3$ such that 
$$
 \left\{
\begin{array}{ll}
-{\rm Div\; } \boldsymbol{\sigma}=\overrightarrow{F_0}, & \mbox{ in }\Omega\times\mathbb{R},\\
\boldsymbol{\sigma}=k_0tr(\boldsymbol{\varepsilon}^D(\overrightarrow{u}))I_3+a(|\boldsymbol{\varepsilon}^D(\overrightarrow{u})|^2)\boldsymbol{\varepsilon}^D(\overrightarrow{u}), & \mbox{ in }\Omega\times \mathbb{R},\\
\overrightarrow{u}=\overrightarrow{0},& \mbox{ on }\Gamma_1\times\mathbb{R},\\
\boldsymbol{\sigma}\overrightarrow{\nu}=\overrightarrow{F_2},& \mbox{ on }\Gamma_2\times\mathbb{R},\\
{\displaystyle |\overrightarrow{\sigma_\tau}|\leq g, \overrightarrow{\sigma_\tau}= -g\frac{\overrightarrow{u_\tau}}{|\overrightarrow{u_\tau}|}} \mbox{ if } \overrightarrow{u_\tau}\neq \overrightarrow{0},&  \mbox{ on }\Gamma_3\times\mathbb{R},
\end{array}
\right.
$$
where $a:(0,\infty)\to(0,\infty)$ is a prescribed function. The first line of $(P_{3D})$ represents the {\it equilibrium equation}, while the second line is the {\it constitutive law}  which describes the behaviour of the material. The last relation is the well known Tresca friction law, with $g:\Gamma_3\to[0,\infty)$ is the friction bound, that is, the threshold from which the slipping begins.  

Loading the body in the following particular way:
$$
\overrightarrow{F_0}(x_1,x_2,x_3):=(0,0,f_0(x_1,x_2)), \mbox{ with }f_0:\Omega\to\mathbb{R},
$$
$$
\overrightarrow{F_2}(x_1,x_2,x_3):=(0,0,f_2(x_1,x_2)), \mbox{ with }f_2:\Gamma_2\to\mathbb{R},
$$
we expect a displacement field $\overrightarrow{u}$ independent of $x_3$ of the form 
$$
\overrightarrow{u}(x_1,x_2,x_3)=(0,0,u(x_1,x_2)), \mbox{ with }u:\bar{\Omega}\to\mathbb{R}.
$$
Combining this with the fact the unit outer normal to $\Gamma\times \mathbb{R}$ is parallel to the plane $Ox_1x_2$, i.e.,   $\overrightarrow{\nu}(x_1,x_2,x_3):=(\nu_1(x_1,x_2),\nu_2(x_1,x_2),0)$, we deduce that the infinitesimal strain tensor $\boldsymbol{\varepsilon}(\overrightarrow u)$ and the stress field $\boldsymbol{\sigma}$ have the following equalities
$$
\boldsymbol{\varepsilon}(\overrightarrow u)=\frac{1}{2}\left(
\begin{array}{ccc}
0 \ & 0 \ &\frac{\partial u}{\partial x_1}\\
0 \ & 0 \ & \frac{\partial u}{\partial x_2}\\
\frac{\partial u}{\partial x_1} &\frac{\partial u}{\partial x_2}  & 0\\
\end{array}
\right), 
\ \boldsymbol{\varepsilon}^D(\overrightarrow{u})=\boldsymbol{\varepsilon}(\overrightarrow{u}), \ \boldsymbol{\sigma}=a(|\nabla u|^2)\boldsymbol{\varepsilon}(\overrightarrow{u}), 
$$
$$
u_\nu=0, \ \overrightarrow{u_\tau}=\overrightarrow{u},\  \boldsymbol{\sigma}\overrightarrow{\nu}=(0,0,a(|\nabla u|^2)\nabla u\cdot \overrightarrow{n}), \sigma_\nu=0, \overrightarrow{\sigma_\tau}=\boldsymbol{\sigma}\overrightarrow{\nu},
$$
where $\overrightarrow{n}=(\nu_1,\nu_2)$ is the outer normal unit to $\Gamma$ in the $Ox_1x_2$ plane.

Now choosing $\phi(t):=a(t^2)t$ and $h(x,t):=-f_0(x)t$ we observe that $h(x,u(x))\in L^{\Omega}$ for any $u\in W^{1,\Phi}(\Omega)$ whenever $f_0\in L^{\Phi^\ast}(\Omega)$. Moreover, 
$$
\partial_2 h(x,t)=\{-f_0(x)\}, \,\forall t\in\mathbb{R} \mbox{
and }
-{\rm Div\; }\boldsymbol{\sigma}=(0,0,-\Delta_\Phi u),
$$
therefore problem $(P_{3D})$ reduces to problem $(P)$ from previous section, provided that the function $a:(0,\infty)\to (0,\infty)$ is such that $t\mapsto a(t^2)t$ satisfies $({\mathcal{H}_1})$. A simple and meaningful example of such function is $a(t):=\gamma \frac{(\sqrt{1+t}-1)^{\gamma-1}}{\sqrt{1+t}}$, with $\gamma> 1$ (see, e.g., Fukagai \& Narukawa \cite{F-N95}).

\pagestyle{plain}
\bibliographystyle{siam}

\bibliography{Biblio}

\begin{thebibliography}{10}

\bibitem{Adams}
{\sc R.~A. Adams}, {\em Sobolev Spaces}, Academic Press, 1975.

\bibitem{And-Cos-Mat}
{\sc I.~Andrei, N.~Costea, and A.~Matei}, {\em Antiplane shear deformation of
  piezoelectric bodies in contact with a conductive support}, J. Glob. Optim.,
  56 (2013), pp.~103--119.

\bibitem{Bai-Mig-Zeng20}
{\sc Y.~Bai, S.~Mig\'{o}rski, and S.~Zeng}, {\em A class of generalized mixed
  variational-hemivariational inequalities {I}: {E}xistence and uniqueness
  results}, Comput. Math. Appl., 79 (2020), pp.~2897--2911.

\bibitem{Brezis2011}
{\sc H.~Brezis}, {\em Functional Analysis, Sobolev Spaces and Partial
  Differential Equations}, Springer, 2011.

\bibitem{B-N-S72}
{\sc H.~Brezis, L.~Nirenberg, and G.~Stampacchia}, {\em A remark on {Ky Fan's}
  minimax principle}, Boll. Unione Mat. Ital., 6 (1972), pp.~293--300.

\bibitem{Clarke}
{\sc F.~H. Clarke}, {\em Optimization and Nonsmooth Analysis}, Classics in
  Applied Mathematics, Society for Industrial and Applied Mathematics, 1990.

\bibitem{C-P-S-T}
{\sc P.~Cl\'{e}ment, B.~de~Pagter, G.~Sweers, and F.~de~Th\'{e}lin}, {\em
  Existence of solutions to a semilinear elliptic system through
  {O}rlicz-{S}obolev spaces}, Mediterr. J. Math., 1 (2004), pp.~241--267.

\bibitem{Cos-Kri-Var}
{\sc N.~Costea, A.~Krist\'{a}ly, and C.~Varga}, {\em Variational and
  Monotonicity Methods in Nonsmooth Analysis}, Frontiers in Mathematics,
  Birkh\"{a}user, Cham, 2021.

\bibitem{Cos-Mat10}
{\sc N.~Costea and A.~Matei}, {\em Weak solutions for nonlinear antiplane
  problems leading to hemivariational inequalities}, Nonlinear Anal., 72
  (2010), pp.~3669--3680.

\bibitem{Cos-Pit21}
{\sc N.~Costea and A.~Pitea}, {\em Existence results for mixed
  hemivariational-like inequalities involving set-valued maps}, Optimization,
  70 (2021), pp.~269--305.

\bibitem{Don-Tru}
{\sc T.~Donaldson and N.~Trudinger}, {\em Orlicz-{S}obolev spaces and imbedding
  theorems}, J. Funct. Anal., 8 (1971), pp.~52--75.

\bibitem{Ekeland-Temam99}
{\sc I.~Ekeland and R.~T\'{e}mam}, {\em Convex Analysis and Variational
  Problems}, vol.~28 of Classics in Applied Mathematics, SIAM, 1999.

\bibitem{F-N95}
{\sc N.~Fukagai and K.~Narukawa}, {\em Nonlinear eigenvalue problem for a model
  equation of an elastic surface}, Hiroshima Math. J., 25 (1995), pp.~19--41.

\bibitem{GH-L-M-S}
{\sc M.~Garc\'{i}a-Huidobro, V.~K. Le, R.~Man\'{a}sevich, and K.~Schmitt}, {\em
  On principal eigenvalues for quasilinear elliptic operators: an
  {O}rlicz-{S}obolev space setting}, NoDEA Nonlinear Differential Equations
  Appl., 6 (1999), pp.~207--225.

\bibitem{HS02}
{\sc W.~Han and M.~Sofonea}, {\em Quasistatic Contact Problems in
  Viscoelasticity and Viscoplasticity}, vol.~30 of Studies in Advanced
  Mathematics, International Press, Somerville, 2002.

\bibitem{Has-Hla-Necas96}
{\sc J.~Haslinger, I.~Hl\'{a}va\v{c}ek, and J.~Ne\v{c}as}, {\em Numerical
  methods for unilateral problems in solid mechanics}, in Finite Element
  Methods (Part 2), Numerical Methods for Solids (Part 2), P.~G. Ciarlet and
  J.-L. Lions, eds., vol.~IV of Handbook of Numerical Analysis, North-Holland,
  1996, pp.~313--485.

\bibitem{Hor95}
{\sc C.~O. Horgan}, {\em Anti-plane shear deformation in linear and nonlinear
  solid mechanics}, SIAM Rev., 37 (1995), pp.~53--81.

\bibitem{Kras-Rut}
{\sc M.~A. Krasnosel'ski\u{\i} and Y.~B. Ruticki\u{\i}}, {\em Convex Functions
  and Orlicz Spaces}, P. Noordhoff Ltd., 1961.

\bibitem{Matei-RWA14}
{\sc A.~Matei}, {\em An existence result for a mixed variational problem
  arising from contact mechanics}, Nonlinear Anal. Real World Appl., 20 (2014),
  pp.~74--81.

\bibitem{Matei15}
\leavevmode\vrule height 2pt depth -1.6pt width 23pt, {\em Two abstract mixed
  variational problems and applications in contact mechanics}, Nonlinear Anal.
  Real World Appl., 22 (2015), pp.~592--603.

\bibitem{Mig-Bai-Zeng19}
{\sc S.~Mig\'{o}rski, Y.~Bai, and S.~Zeng}, {\em A class of generalized mixed
  variational-hemivariational inequalities {II}: {A}pplications}, Nonlinear
  Anal. Real World Appl., 50 (2019), pp.~633--650.

\bibitem{M-O-S}
{\sc S.~Migorski, A.~Ochal, and M.~Sofonea}, {\em Weak solvability of antiplane
  frictional contact problems for elastic cylinders}, Nonlinear Anal. Real
  World Appl., 11 (2010), pp.~172--181.

\bibitem{mosco}
{\sc U.~Mosco}, {\em Implicit variational problems and quasi-variational
  inequalities}, in Nonlinear Operators and the Calculus Of Variations, J.~P.
  Gossez, E.~J. {Lami Dozo}, J.~Mawhin, and L.~Waelbroek, eds., Lecture Notes
  in Mathematics 543, Springer-Verlag, 1976, pp.~83--156.

\bibitem{Mat-Sof09}
{\sc M.~Sofonea and A.~Matei}, {\em Variational Inequalities with Applications.
  A Study of Antiplane Frictional Contact Problems}, vol.~18 of Advances in
  Mechanics and Mathematics, Springer, New York, 2009.

\end{thebibliography}

\end{document}